\documentclass[12pt,reqno]{amsart}
\usepackage{cite}
\usepackage[colorlinks,citecolor=blue,hypertexnames=false, urlcolor=blue]{hyperref}
\usepackage{amssymb,amsmath,amsfonts,latexsym}
\usepackage{bm}
\usepackage[us,12hr]{datetime}
\usepackage{mathtools, mathrsfs}
\allowdisplaybreaks
\usepackage{enumerate}
\usepackage{fancyhdr}
\usepackage{soul}

\usepackage{array,graphics,xcolor}
\numberwithin{equation}{section}
\newtheorem{dfn}{Definition}[section]
\newtheorem{thm}[dfn]{Theorem}
\newtheorem{lma}[dfn]{Lemma}

\newtheorem{ppsn}[dfn]{Proposition}
\newtheorem{cor}[dfn]{Corollary}

\newcommand\at[3]{#1\,#2\Big|_{#3}}

\newcommand{\R}{\mathbb{R}}
\newcommand{\N}{\mathbb{N}}
\newcommand{\C}{\mathbb{C}}

\newcommand{\Scal}{\mathcal{S}}
\newcommand{\Hcal}{\mathcal{H}}

\newcommand{\Rcal}{\mathcal{R}}
\newcommand{\Wcal}{\mathcal{W}}
\newcommand{\bh}{\mathcal{B}(\Hcal)}
\newcommand{\Tr}{\operatorname{Tr}}

\begin{document}
	\title[Positivity of SSF and infinite-dimensional BMV conjecture]
	{Positivity of spectral shift functions and infinite-dimensional BMV conjecture}
	
	\author[Pradhan]{Chandan Pradhan}
	\address{C. P., Department of Mathematics and Statistics, MSC01 1115, University of New Mexico, Albuquerque, NM 87131, USA}
	\email{chandan.pradhan2108@gmail.com, cpradhan@unm.edu}
	
	\author[Skripka]{Anna Skripka}
	\address{A. S., Department of Mathematics and Statistics, MSC01 1115, University of New Mexico, Albuquerque, NM 87131, USA}
	\email{askripka@unm.edu}
	
	
	
	\begin{abstract}
	We obtain a solution to the Bessis-Moussa-Villani conjecture for a trace-class perturbation of a semi-bounded operator and answer affirmatively the question on positivity of higher order spectral shift functions in the setting of Schatten--von Neumann perturbations of (possibly unbounded) self-adjoint operators.
		
	\end{abstract}
	
	\maketitle
	
	\section{Introduction}\label{sec:intro}

 In this paper we answer two seemingly unrelated questions in mathematical physics by developing a unified approach that may also apply to other problems. The first question is on the infinite-dimensional BMV conjecture and the second one is on the positivity preserving property of higher order spectral shift functions.

In their work on quantum mechanical systems \cite{BMV}, D.~Bessis, P.~Moussa, and M.~Villani studied the partition function of a perturbed semi-bounded operator and conjectured that the partition function, regarded as a function of the coupling constant, is the Laplace transform of a positive measure (see \cite[(I.3)]{BMV}). The finite-dimensional version of this long-standing problem was ultimately resolved in the celebrated work \cite{Stahl}, where it was shown that for finite self-adjoint matrices $H$ and $V\ge 0$, there exists a positive measure $\nu$ on $[0,\infty)$ such that \begin{align}
\label{lir}
\Tr(e^{H-tV})=\int_{[0,\infty)}e^{-ts}d\nu(s).
\end{align}

If $H-tV$ is an infinite-dimensional operator with essential spectrum, then $\Tr(e^{H-tV})$ is not well defined and a correcting term needs to be subtracted. We establish that if $H$ is bounded from below and $V\ge 0$ is trace-class, then $t\mapsto\Tr( e^{-H}-e^{-H-tV})$ admits a L\'{e}vy-Khintchine representation constructed from a positive measure (see Corollary \ref{cor:BMV-1}). It is a particular case of a more general result established in Theorem \ref{thm:BMV-unbdd} for the function $t\mapsto\Tr(f(H+tV)-f(H))$, where $f'$ is completely monotone and $f(H)$ is defined by the standard functional calculus.
We also obtain a result on a variant of the BMV conjecture for non-trace-class $V$ (see Theorem~\ref{bmvsn}) inspired by results of \cite{LiSe12} for finite-dimensional operators.

To discuss spectral shift functions we need to introduce some notation.
Let $\Hcal$ be a separable complex Hilbert space, $\bh$ the algebra of bounded linear operators on $\Hcal$, $n\in\N$, and $\Scal^n$ the $n$th Schatten--von Neumann ideal of operators in $\bh$. Let $H$ be a self-adjoint operator densely defined in $\Hcal$, which we call for brevity ``a self-adjoint operator defined in $\Hcal$".	
	Let $n\in\N$, $V =V^*\in \Scal^n$, and $f \in  \bigcap_{k=0}^n \Wcal_k(\R)$ (see \eqref{defn-Wn}). Denote by $\Rcal_n(f,H,V)$ the $n$th-order Taylor remainder given by
	\begin{align}\label{eq:rem}
		\Rcal_n(f,H,V)
		= f(H+V) - \sum_{k=0}^{n-1} \frac{1}{k!} \at{\frac{d^k}{dt^k}}{f(H+tV)}{t=0},
	\end{align}
	with G\^{a}teaux derivatives of operator functions calculated in the operator norm. The existence of the latter derivatives is justified in Theorem \ref{thm:differetiation}.
Let $\Omega \subset \mathbb{R}$ be an interval and let $G_\Omega$ denote the closure of the convex hull of the union of the spectra of $H+tV$, that is,
\begin{align}
\label{def:m-M}
G_\Omega=\overline{{\rm c.v.h.}\big(\cup_{t\in\Omega}\sigma(H + tV)\big)}.
\end{align}

	The following fundamental trace formulas  were established in \cite{Krein}, \cite{Koplienko}, and \cite{PSS} for $n=1,2$ and for $n\ge 3$, respectively.
	
	\begin{thm}\label{thm:existancessf}
		Let $n \in \N$, and let $H, V$ be self-adjoint operators in $\Hcal$ such that $V \in \Scal^n$. Then, there exists a unique $\eta_{n,H,V} \in L^1(\R)$ such that
		\begin{align}\label{traceformula_inv}
			\Tr\big(\Rcal_n(f,H,V)\big)
			= \int_\R f^{(n)}(\lambda)\,\eta_{n,H,V}(\lambda)\, d\lambda
		\end{align}
		for every $f \in  \bigcap_{k=0}^n \Wcal_k(\R)$.
 The function $\eta_{n,H,V}$ is supported in the set $G_{[0,1]}$ defined in \eqref{def:m-M}.
	\end{thm}
	
	The function $\eta_{n,H,V}$ satisfying the trace formula \eqref{traceformula_inv} is called the spectral shift function of order $n$.
	It is well known that $\eta_{2,H,V} \ge 0$ and that for a sign-definite perturbation $V$ the function $\eta_{1,H,V}$ has the same sign as $V$. The latter properties can be derived from \cite{BiSo73} and their proofs can also be found in \cite[Theorem 4.3]{ChSi12}, \cite[Section~8.2, Theorem~1]{Yafaev92}, and \cite[Proposition 3.3]{Sk17}.
	The sign-definiteness of $\eta_{n,H,V}$ with $n\ge 3$ was also investigated, but apart from the partial cases treated in \cite[Theorem 1.4, Corollary~1.6, Proposition~3.4]{Sk17}, the question remained open in full generality due to the complexity of higher order spectral shift functions.
	The latter can be seen from the explicit formulas for $\eta_{n,H,V}$, $n\ge 3$, derived in \cite[Theorem 5.1(iii)]{ds} for $H\in\bh$ and $V\in\Scal^2$.

 In Theorem \ref{thm:positivity-ssf} we establish that all spectral shift functions of even order are nonnegative, and that those of odd order are also nonnegative (respectively, nonpositive) whenever the perturbation is nonnegative (respectively, nonpositive).
	
Our approach to both main results is based on combining the powerful method of multilinear operator integration and the recent groundbreaking work \cite{Otte}.
As a part of our method, we extend the result of \cite[Corollary 1.2]{Otte} from finite-dimensional to infinite-dimensional operators, with a suitable modification of the formula to ensure that the trace is well defined (see Theorem~\ref{thm:Ext-otte} and Theorem \ref{thm:Ext-otte-unbdd}). The positivity preserving property of the spectral shift functions is applied in the proof of our results on the infinite-dimensional BMV conjecture.
	
	In Section~\ref{sec:Preliminaries} we introduce additional notation and collect preliminaries on multilinear operator integration, in Section~\ref{sec:finite} we discuss the positivity preserving property of the trace of an operator derivative, in Section~\ref{sec_pos_ssf} we establish the positivity preserving property of higher order spectral shift functions, and in Section \ref{sec:bmv} we establish the results on the BMV conjecture in the infinite-dimensional case.

	\section{Multilinear Operator Integration}\label{sec:Preliminaries}
	
	In this section we collect fundamental properties of multilinear operator integrals used in the proof of our main results. A more detailed discussion of the subject can be found in \cite{ST19}.

Let $\N_0=\N\cup \{0\}$. Let $\Omega\subset\R$ be an interval, where $\subset$ denotes a nonstrict set inclusion.
Let $C^0(\Omega):=C(\Omega)$, the space of all continuous functions on $\Omega$. For $n\in\N$, let $C^n(\Omega)$ and $C^\infty(\Omega)$ denote the spaces of all $n$-times continuously differentiable functions and infinitely many times continuously differentiable functions on $\Omega$, respectively. Let $C_c^n(\R)$ denote the space of compactly supported $n$-times continuously differentiable functions on $\R$.
	
	Recall that the divided difference of $f\in C^n(\Omega)$ of order $n$ is
	defined recursively as follows:
	\begin{align*}
		&f^{[0]}(\lambda)=f(\lambda),\\
		&f^{[n]}(\lambda_0,\lambda_1,\ldots,\lambda_n)=\begin{cases}
			\frac{f^{[n-1]}(\lambda_0,\lambda_1,\ldots,\lambda_{n-2},\lambda_n)-f^{[n-1]}(\lambda_0,\lambda_1,\ldots,\lambda_{n-2},\lambda_{n-1})}{\lambda_n-\lambda_{n-1}} \quad &\text{if}\quad \lambda_n\neq\lambda_{n-1},\\
			\frac{\partial}{\partial \lambda}f^{[n-1]}(\lambda_0,\lambda_1,\ldots,\lambda_{n-2},\lambda)\big|_{\lambda=\lambda_{n-1}}\quad &\text{if}\quad \lambda_n=\lambda_{n-1}.
		\end{cases}
	\end{align*}

Let $L^p(\R)$ denote the standard $L^p$-space on $\R$ with respect to the  Lebesgue measure.
For $n\in\N_0$, let $\Wcal_n(\R)$ denote the Wiener class given by
\begin{align}\label{defn-Wn}
\Wcal_n(\R)
&=\{f\in C^n(\R):\; f^{(n)}\in L^\infty(\R),\, \widehat{f^{(n)}}\in L^1(\R)\},
\end{align}
where the Fourier transform of a nonintegrable function is understood in the sense of tempered distributions. By a standard exercise in Fourier analysis,
\begin{align*}
\{f\in C^{ n+1}(\R):\; f^{(n)}, f^{(n+1)}\in L^2(\R)\}\subset\Wcal_n(\R).
\end{align*}
In particular, we also have $C_c^{n+1}(\R)\subset \Wcal_n(\R)$.

The following definition of the multilinear operator integral is due to \cite{PSS}. It is stated for a smaller set of admissible functions $f$ sufficient for the purpose of this paper. We adopt the standard convention $\Scal^\infty=\bh$.
	
	\begin{dfn}\label{def:moi}
		Let $n\in\N$, let $H_0,H_1,\ldots,H_n$ be self-adjoint operators in $\Hcal$, and let $f\in \Wcal_n(\R)$.
		Let $\alpha,\alpha_k\in[1,\infty]$,  $k=1,\dots,n$, satisfy $\frac{1}{\alpha}=\frac{1}{\alpha_1}+\frac{1}{\alpha_2}+\cdots+\frac{1}{\alpha_n}$ and let $E^i_{l,m}= E_{H_i}\big(\big[\frac{l}{m},\frac{l+1}{m}\big)\big)$ for $m\in \mathbb{N}$ and $l\in \mathbb{Z}$, where $E_{H_i}$ is the spectral measure of $H_i$, $i=0,\dots,n$. Define a multilinear transformation on $\Scal^{\alpha_1}\times\cdots\times\Scal^{\alpha_n}$ by
		\begin{align}\label{a1}
			&T^{H_0,\ldots,H_n}_{f^{[n]}}(V_1,V_2,\ldots,V_n)\\
			\nonumber
			&=\lim_{m\to\infty}\lim_{N\to\infty}\sum_{|l_0|,|l_1|,\ldots,|l_n|\leq N}f^{[n]}\Big(\frac{l_0}{m},\frac{l_1}{m},\ldots,\frac{l_n}{m}\Big)E_{l_0,m}^0V_1E_{l_1,m}^1
			V_2E_{l_2,m}^2\cdots V_n E_{l_n,m}^n,
		\end{align}
where $\lim\limits_{N\to\infty}$ is evaluated in the norm $\|\cdot\|_\alpha$ when $\alpha<\infty$ and in the strong operator topology when $\alpha=\infty$; $\lim\limits_{m\to\infty}$ is evaluated in $\|\cdot\|_\alpha$. The existence of the limits in \eqref{a1} is justified in \cite[Lemmas 3.5, 5.1, 5.2]{PSS}, and the transformation $T^{H_0,\ldots,H_n}_{f^{[n]}}$ is called a multilinear operator integral
		with symbol $f^{[n]}$.
	\end{dfn}

	The following estimate is an extension of the result of \cite[Theorem 5.3 and Remark 5.4]{PSS} for $H_0=\cdots=H_n$ to the case of distinct $H_0,\dots,H_n$, which proof is discussed in \cite[Theorem 4.3.10]{ST19}.
	\begin{thm}\label{inv-thm}
		Let $n\in\mathbb{N}$ and let $\alpha,\alpha_1,\ldots,\alpha_n\in(1,\infty)$ satisfy $\tfrac{1}{\alpha_1}+\cdots+\tfrac{1}{\alpha_n}=\tfrac{1}{\alpha}$. Let $H_0,\ldots, H_{n}$ be self-adjoint operators in $\Hcal$. Assume that $V_\ell\in\mathcal{S}^{\alpha_\ell},\, 1\leq \ell\leq n$. Then, there exists a positive constant $c_{\alpha,n}$, depending only on  $n,\alpha_1,\ldots,\alpha_n$, such that
		\begin{align}\label{est}
			\big\|T^{H_0,\ldots,H_{n}}_{f^{[n]}}(V_1,V_2,\ldots,V_n)\big\|_{\alpha}\leq c_{\alpha,n}\, \|f^{(n)}\|_\infty \prod\limits_{1\leq \ell \leq n}\|V_\ell\|_{\alpha_\ell}
		\end{align}
		for every $f\in C_c^{n+1}(\R)$.
	\end{thm}
	
The result of Theorem \ref{inv-thm} extends to multiliear operator integrals with symbol $$(\lambda_0,\ldots,\lambda_{n-1})\mapsto f^{[n]}(\lambda_0,\ldots,\lambda_{n-1},\lambda_0).$$
Applying it along with the cyclicity of the trace and H\"{o}lder's inequality yields the bound for the trace of a multilinear operator integral given below.

For an operator $T$ and a nonnegative integer $k$, we define
	\[
	(T)^k :=
	\begin{cases}
		\underbrace{T, \dots, T}_{k \text{ times}}, & \text{if } k \neq 0,\\[2mm]
		\varnothing, & \text{if } k = 0,
	\end{cases}
	\]
	where $\varnothing$ denotes the empty tuple. Throughout this article, we will use this notation frequently.
	
	\begin{cor}\label{thm:tr-est}
		Let $n\in\mathbb{N}$, $n\geq 3$, and let $\alpha_1,\ldots,\alpha_n\in(1,\infty)$ satisfy $\tfrac{1}{\alpha_1}+\cdots+\tfrac{1}{\alpha_n}=1$. Let $H, K, L$ be self-adjoint operators in $\Hcal$. Assume that $V_\ell\in\Scal^{\alpha_\ell},\, 1\leq \ell\leq n$. Then, there exists $c_n>0$, depending on $n,\alpha_1,\ldots,\alpha_n$, such that
		\begin{align}\label{ss}
			\big|\Tr(T^{H,K,L,(H)^{n-2}}_{f^{[n]}}(V_1,\ldots,V_n))\big|&\leq c_{n}\|f^{(n)}\|_\infty\prod\limits_{1\leq \ell \leq n}\|V_\ell\|_{\alpha_\ell}
		\end{align}
		for every $f\in C_c^{n+1}(\R)$.
	\end{cor}
	
	We will also need the following non-Schatten version of the perturbation formula for multilinear operator integrals (see, e.g., \cite[Theorem 4.3.14]{ST19}). The corresponding Schatten version was obtained in \cite[Lemma~4.2]{LeSk20}.

	\begin{lma}\label{lem:moi_pert_form}
		Let $n \in \N$, $n \ge 2$, and $f \in \Wcal_{n-1}(\R) \cap \Wcal_n(\R)$.
		Let $H, K, H_1, \ldots, H_{n-1}$ be self-adjoint operators in $\Hcal$ such that $H - K \in \bh$,
		and let $V_1, \ldots, V_n \in \mathcal{B}(\Hcal)$.
		Then for every $i = 1, \ldots, n$,
		\begin{align*}
			& T^{H_1, \ldots, H_{i-1}, H, H_i, \ldots, H_{n-1}}_{f^{[n-1]}}
			(V_1, V_2, \ldots, V_{n-1})
			- T^{H_1, \ldots, H_{i-1}, K, H_i, \ldots, H_{n-1}}_{f^{[n-1]}}
			(V_1, V_2, \ldots, V_{n-1}) \\
			&\qquad =
			T^{H_1, \ldots, H_{i-1}, H, K, H_i, \ldots, H_{n-1}}_{f^{[n]}}
			(V_1, \ldots, V_{i-1}, H - K, V_i, \ldots, V_{n-1}).
		\end{align*}
	\end{lma}
	
	\smallskip
	
	The following continuity result is needed in the sequel.
	
	\begin{lma}\label{lem:moi-cont}
		Let $n \in \N$.
		Let $H_0, \ldots, H_n, H_{0k}, \ldots, H_{nk}$ be self-adjoint operators in $\Hcal$ such that $H_{ik} \to H_i$ resolvent strongly as $k\rightarrow\infty$ for $i=0,\dots,n$ and let $V_k, V \in \Scal^n$ be self-adjoint operators such that $\|V_k-V\|_n\rightarrow 0$ as $k\rightarrow\infty$. Then, for each $f \in \Wcal_n(\R)$,
		\begin{align*}
			\big\| T_{f^{[n]}}^{H_{0k},\ldots,H_{nk}}\big( (V_k)^n \big)
			- T_{f^{[n]}}^{H_0, \ldots, H_n}\big( (V)^n \big) \big\|_1 \to 0
		\end{align*}
		as $k \to \infty$.
	\end{lma}
	
	\begin{proof}
		By  \cite[Lemmas 3.5, 5.1, 5.2]{PSS},
		\begin{align}
			T_{f^{[n]}}^{H_{0k},\ldots,H_{nk}}\big( (V_k)^n \big)
			&= \int_{\R} \int_{\Delta_n}
			e^{is_0t H_{0k}} V_k e^{is_1t H_{1k}} V_k \cdots V_k e^{is_n t H_{nk}} \,
			\widehat{f^{(n)}}(t) \, d\sigma\,dt,
		\end{align}
		where the simplex
		\[\Delta_n=\{s=(s_0,\ldots,s_n)\in[0,\infty)^{n+1}:\; s_0+\cdots+s_n=1\}\]
		is endowed with the Lebesgue measure $\sigma$. Note that $\nu_f\times\sigma$, where $d\nu_f(t)=\widehat{f^{(n)}}(t)\,dt$, is a finite measure on $\R\times\Delta_n$ with total variation $\frac{1}{n!}\|\widehat{f^{(n)}}\|_1$.
		
		Since $H_{jk}$ converges resolvent strongly to $H_j$, it follows from \cite[Theorem~VIII.20(b)]{ReedSimonI} that $e^{i  s_j t H_{jk}} \to e^{i s_j t H_j}$  in the strong operator topology for every $j=0,\dots,n$.
Since we also have $\|V_k-V\|_n\rightarrow 0$, by Gr\"{u}mm's well-known convergence theorem \cite[Theorem 1]{G}, we obtain that $\|e^{i s_j t H_{jk}} V_k-e^{i s_j t H_j} V\|_n\to 0$ and $ \|Ve^{i s_j t H_{jk}}-Ve^{i s_j t H_j}\|_n\to 0$. Moreover, there exists a constant $L > 0$ such that $\sup_k \bigl\{ \|V_k\|_n, \|V\|_n \bigr\} \le L$. Therefore, by telescoping, the triangle inequality and H\"older’s inequalities for Schatten norms, we obtain
		\begin{align*}
			&\Big\| e^{is_0 t H_{0k}} V_k e^{is_1 t H_{1k}} V_k \cdots V_k e^{is_n t H_{nk}}
			- e^{is_0 t H_0} V e^{is_1 t H_1} V \cdots V e^{is_n t H_n} \Big\|_1 \\
			=& \Big\| \big(e^{is_0 t H_{0k}} V_k - e^{is_0 t H_{0}} V\big) e^{is_1 t H_{1k}} V_k \cdots V_k e^{is_n t H_{nk}}\\
			&\hspace*{1.2in} + e^{is_0 t H_{0}} V \big(e^{is_1 t H_{1k}} V_k - e^{is_1 t H_{1}} V\big)e^{is_2 t H_{2k}} V_k \cdots V_k e^{is_n t H_{nk}}\\
			&\hspace*{1.5in}+\cdots+ e^{is_0 t H_{0}} V\cdots e^{is_{n-2} t H_{n-2}}V \big(e^{is_{n-1} t H_{(n-1)k}} V_k - e^{is_{n-1} t H_{n-1}} V\big)e^{is_n t H_{nk}}\\
			&\hspace*{2in}  +  e^{is_0 t H_{0}} V\cdots V e^{is_{n-1} t H_{n-1}} \big(V e^{is_{n} t H_{nk}} - V e^{i s_n t H_{n}} \big)\Big\|_1\\
			\leq& L^{n-1}\sum_{j=0}^{n-1} \Big\|e^{is_j t H_{jk}} V_k - e^{is_j t H_{j}} V\Big\|_n + L^{n-1} \Big\|V e^{is_{n} t H_{nk}} - V e^{i s_n t H_{n}}\Big\|_n\to 0
		\end{align*}
		as $k\to \infty$ and
		\begin{align*}
			\| e^{is_0 t H_{0k}} V_k e^{is_1 t H_{1k}} V_k \cdots V_k e^{is_n t H_{nk}}
			- e^{is_0 t H_0} V e^{is_1 t H_1} V \cdots V e^{is_n t H_n} \|_1\leq 2L^n.
		\end{align*}
		Finally, an application of the dominated convergence theorem
		for Bochner integrals completes the proof.
	\end{proof}

	\smallskip
	
	\begin{thm}\label{thm:differetiation}
Let $n \in \N$, let $H, V$ be self-adjoint operators in $\Hcal$ such that $V \in \Scal^n$, and let $ f \in \cap_{k=1}^n\Wcal_k(\R)$.
Then, the map $\R \ni t \mapsto f(H + tV)$ is $k$-times differentiable in the operator norm and
		\begin{align*}
			\frac{1}{k!}
			\at{\frac{d^{k}}{dt^{k}}}{f(H + tV)}{t = s}
			= T_{f^{[k]}}^{(H + sV)^{k+1}}\!\big((V)^k\big)
		\end{align*}
for each $k =1,\ldots, n$.
		Moreover, the map $s \mapsto\frac{d^{n-1}}{dt^{n-1}}f(H + tV)\big|_{t = s}$
		is differentiable in the norm $\|\cdot\|_1$ and
		\begin{align*}
			\frac{1}{n!}
			\at{\frac{d^{n}}{dt^{n}}}{f(H + tV)}{t = s}
			= T_{f^{[n]}}^{(H + sV)^{n+1}}\!\big((V)^n\big).
		\end{align*}
	\end{thm}
	
	\begin{proof}
		The $k$-times differentiability of $
		\R \ni t \mapsto f(H + tV)$
		in the operator norm was established in \cite{Pe06} and summarized in \cite[Theorem~5.3.5]{ST19}.
		The differentiability of
		\[
		s \mapsto
		\at{\frac{d^{n-1}}{dt^{n-1}}}{f(H + tV)}{t = s}
		\]
		in the norm $\|\cdot\|_1$ can be proved along the lines of the proof of \cite[Theorem~5.3.5]{ST19} with help of Lemma~\ref{lem:moi-cont}.
	\end{proof}

 The operator function $t\mapsto f(H+tV)$ is known to be differentiable for a larger class of functions than the one considered in Theorem \ref{thm:differetiation}. However, the property $f\in C^1(\R)$ is insufficient for the differentiability of $t\mapsto f(H+tV)$ in the operator norm even if $H$ is bounded. A detailed discussion of the differentiability of operator functions can be found in \cite[Section 5.3]{ST19}.
	
	\begin{lma}\label{lem:pert_form}
		Let $n \in \N$, $n \geq 2$, and $f \in C_c^{n+1}(\R)$. Let $H, V$ be self-adjoint operators in $\Hcal$ with $V \in \Scal^n$. Then,
		\begin{align}\label{eq:per0}
			\Rcal_n(f, H, V)
			&= T_{f^{[n-1]}}^{H, H+V, (H)^{n-2}}((V)^{n-1})
			- T_{f^{[n-1]}}^{(H)^n}((V)^{n-1}),
		\end{align}
		where $\Rcal_n(f, H, V)$ is given by~\eqref{eq:rem}.
	\end{lma}
	
	\begin{proof}
		By Theorem \ref{thm:differetiation}, we have
		\begin{align}\label{eq:0}
			\Rcal_n(f, H, V)
			&= f(H+V) - f(H)
			- \sum_{k=1}^{n-1}
			T_{f^{[k]}}^{(H)^{k+1}}((V)^k).
		\end{align}
		By the well-known Birman--Solomyak perturbation formula \cite[Consequence of Theorem 4.5]{BiSo73}, we have
		\begin{align}\label{eq:1}
			f(H+V) - f(H)
			&= T_{f^{[1]}}^{H+V, H}(V)
			= T_{f^{[1]}}^{H, H+V}(V).
		\end{align}
		Applying~\eqref{eq:1} and Lemma~\ref{lem:moi_pert_form} in \eqref{eq:0} repeatedly implies~\eqref{eq:per0}.
	\end{proof}

	\section{ Positivity preserving property of operator derivatives}\label{sec:finite}
	
	In this section, we establish the positivity preserving property of the trace of a higher order operator derivative in the infinite-dimensional setting.
	
The starting point for our results is the following positivity preserving property of the trace of a matrix function.
	
\begin{thm}{\rm(\!\!\cite[Corollary 1.2]{Otte})}\label{thm:otte}
Assume that $\dim(\Hcal) < \infty$ and let $H, V$ be self-adjoint operators on $\Hcal$.
Let $n\in\N$, $\Omega\subset\R$ be an interval, $f \in C^n(G_\Omega)$, where $G_\Omega$ is given by \eqref{def:m-M}, and consider the function $\phi:\Omega\to\C$ defined by \[\phi(t)=\Tr\!\left(f(H + tV)\right).\]
If $f^{(n)} \ge 0$ on $ G_\Omega$, then the following assertions hold.
\begin{enumerate}[(i)]
\item If $n$ is even, then $\phi^{(n)} \ge 0$.
			
\item If $n$ is odd, then $\phi^{(n)} \ge 0$ (respectively, $\phi^{(n)} \le 0$) provided $V \ge 0$ (respectively, $V \le 0$).
\end{enumerate}
\end{thm}
	
	It is well known that the function $\phi$ defined in Theorem \ref{thm:otte} is differentiable $n$ times (see, e.g., \cite[Theorem 5.3.2]{ST19}). If $t\in\Omega$ is an endpoint of the interval $\Omega$, then the derivative at $t$ is defined using the one-sided limit. It follows from \cite[Theorem 5.3.2]{ST19} that the derivative of $\phi$ can also be computed by the formula
	\begin{align}
		\label{trder}
		\phi^{(n)}(s)
		&= \Tr\!\left(\frac{d^n}{dt^n} f(H + tV)\Big|_{t = s}\right).
	\end{align}
	While the trace of $f(H+tV)$ is generally undefined for infinite-dimensional operators $H$ and $V$, the trace of $\frac{d^n}{dt^n}f(H + tV)\big|_{t = s}$ is defined for a broad class of functions $f$ when $V\in\Scal^n$ (see Theorem~\ref{thm:differetiation}) and the right-hand side of \eqref{trder} becomes a natural replacement of $\phi^{(n)}$ in the infinite-dimensional case. The latter observation explains the modification of the statement of Theorem~\ref{thm:otte} when the result is extended to the setting of an infinite-dimensional Hilbert space $\Hcal$ in Theorem~\ref{thm:Ext-otte}  and Theorem~\ref{thm:Ext-otte-unbdd}.
	
We will need the following approximation result.
	
\begin{lma}\label{lem:finite_proj}
Let $n \in \N$. Let $H, V \in \bh$ be self-adjoint operators such that $V \in \Scal^n$. Then, there exists a sequence $\{P_k\}$ of finite-rank projections, strongly convergent to the identity, such that
\begin{enumerate}[(i)]
\item $\|P_k^{\perp} V\|_n = \|V P_k^{\perp}\|_n \to 0$ as $k \to \infty$,
\item $P_k H P_k \to H$ in the strong operator topology as $k \to \infty$.
\end{enumerate}
\end{lma}
	
	\begin{proof}
By the Weyl–von Neumann–Kuroda theorem (see, e.g., \cite[Chapter 10, Theorem~2.3]{Kato}), there exist a bounded sequence of real numbers
		$\{\lambda_i\}_{i \in \N}$, an orthonormal basis $\{e_i\}_{i \in \N}$ for $\Hcal$,
		and an operator $X \in \Scal^{n+1}$ such that
		\begin{align*}
			H&=\sum_{i=1}^{\infty} \lambda_i \langle \cdot, e_i \rangle e_i + X.
		\end{align*}
		Let $P_k$ denote the orthogonal projection onto $\operatorname{span}\{e_i\}_{i=1}^k$. Then, $P_k \uparrow I$ (i.e., $P_k$ increases strongly to the identity), implying (i). Since
		\begin{align*}
			\sum_{i=1}^{k} \lambda_i \langle \cdot, e_i \rangle e_i
			\longrightarrow\sum_{i=1}^{\infty} \lambda_i \langle \cdot, e_i \rangle e_i
			\quad \text{strongly as } k \to \infty,
		\end{align*}
(ii) also follows.
	\end{proof}

	\begin{thm}\label{thm:Ext-otte}
		Let $n\in \N$, let $H\in\bh$ and $V \in \Scal^n$ be self-adjoint operators.
		Let $\Omega\subset\R$ be an interval, $f \in C^{n+1}(G_\Omega)$, where $G_\Omega$ is given by \eqref{def:m-M}, and consider the function $\psi:\Omega\to\C$ defined by
\begin{align*}
\psi(s)=\Tr\!\left(\frac{d^n}{dt^n} f(H + tV)\Big|_{t=s}\right).
\end{align*}
		If $f^{(n)} \ge 0$ on $ G_\Omega$, then the following assertions hold.
		\begin{enumerate}[(i)]
			\item If $n$ is even, then $\psi \ge 0$.
			
			\item If $n$ is odd, then $\psi \ge 0$ (respectively, $\psi \le 0$) provided $V \ge 0$ (respectively, $V \le 0$).
		\end{enumerate}
		
	\end{thm}
	
	\begin{proof}
Firstly we justify that $\psi$ is well defined. For every $s\in\Omega$, let $\widetilde\Omega_s\subset\Omega$ be a bounded subinterval containing $s$. Then, $G_{\widetilde\Omega_s}\subset G_\Omega$ is a bounded closed interval and there exists $g\in C_c^{n+1}(\R)$ such that $g|_{ G_{\widetilde\Omega_s}}=f|_{ G_{\widetilde\Omega_s}}$. Applying Theorem \ref{thm:differetiation} and the equality $g(H+tV)=f(H+tV)$ for every $t\in\widetilde\Omega_s$ confirms that $\frac{d^n}{dt^n}f(H+tV)\big|_{t=s}=\frac{d^n}{dt^n}g(H+tV)\big|_{t=s}$ exists and belongs to $\Scal^1$. Moreover, by Theorem~\ref{thm:differetiation}, for every $t \in \widetilde\Omega_s$, we have
	\[\psi(t)=n!\,\Tr\!\left(T_{g^{[n]}}^{(H+tV)^{n+1}}\!\big((V)^n\big)\right).\]
	
	(i) Let $\{P_k\}$ be a sequence of finite-rank projections satisfying Lemma~\ref{lem:finite_proj}. Then, the operators
\[H_k = P_k H P_k, \quad V_k = P_k V P_k\]
act on the finite-dimensional Hilbert space $P_k(\Hcal)$. For every $t\in \widetilde\Omega_s$,
$$\overline{\rm{c.v.h}\,\left( \sigma(P_k(H+tV)\big|_{P_k(\Hcal)})\right)}\subseteq G_{\widetilde\Omega_s}.$$
Moreover, for all $t, t'\in \widetilde{\Omega}_s$,
\begin{align*}
&\Tr\,\big(g(H_k+tV_k)-g(H_k+t'V_k)\big)\\
=& \Tr\,\Big( \big( g(P_k(H+tV)\big|_{P_k(\Hcal)})\oplus g(0)P_k^\perp \big) - \big( g(P_k (H+t'V)\big|_{P_k(\Hcal)})\oplus g(0)P_k^\perp \big) \Big)\\
=&\Tr\,\Big( f(P_k(H+tV)\big|_{P_k(\Hcal)}) - f(P_k (H+t'V)\big|_{P_k(\Hcal)}) \Big).
\end{align*}
	
Define $\phi_k : \widetilde\Omega_s \to \C$ by
\[\phi_k(t)=\Tr\!\big( f(P_k(H+tV)\big|_{P_k(\Hcal)})\big).\]
Fix $s_0\in\widetilde{\Omega}_s$.  By Theorem~\ref{thm:differetiation} and the above observations,
\begin{align*}
\phi_k^{(n)}(s_0)
&=\frac{d^n}{dt^n}\Tr\!\left(f\big(P_k(H+tV)\big|_{P_k(\Hcal)}\big)-f\big(P_k(H+s_0V) \big|_{P_k(\Hcal)}\big)\right)\Big|_{t=s_0}\\
&=\frac{d^n}{dt^n}\Tr\!\big(g(H_k+tV_k)-g(H_k+s_0V_k)\big)\Big|_{t=s_0}\\
&=\Tr\!\left(\frac{d^n}{dt^n}\, g(H_k+tV_k)\Big|_{t=s_0}\right)\\
&=n!\,\Tr\!\left(T_{g^{[n]}}^{(H_k + s_0V_k)^{n+1}}\!\big( (V_k)^n \big)\right).
\end{align*}
	By Lemma~\ref{lem:finite_proj}, $\|V_k-V\|_n\to 0$ and $H_k \to H$ strongly and, hence, also in the strong resolvent sense.
	Therefore, for each $s \in \Omega$, Lemma~\ref{lem:moi-cont} implies that $\phi_k^{(n)}(s) \to \psi(s)$. Since $\phi_k^{(n)}(s) \geq 0$ for all $k$ by Theorem~\ref{thm:otte}, we also obtain $\psi(s) \geq 0$ on $\Omega$.
\smallskip
		
(ii) If $V\ge 0$ (respectively, $V\le 0$), then $V_k \geq 0$ (respectively, $V_k\le 0$) for all $k\in\N$. The rest of the proof goes along the lines of the proof of (i).
	\end{proof}

\begin{thm}\label{thm:Ext-otte-unbdd}
Let $m\in\R$, $n\in\N$, $H$ be a self-adjoint operator in $\Hcal$ such that $H\ge mI$, and let $V\in\Scal^n$ satisfy $V\ge 0$. Let $g\in\cap_{j=1}^n\Wcal_j(\R)$ and consider the function $\psi:[0,\infty)\to \C$ defined by
\begin{align*}
\psi(s)=\Tr\!\left(\frac{d^n}{dt^n} g(H + tV)\Big|_{t=s}\right).
\end{align*}
If $g^{(n)}|_{G_{[0,\infty)}}\ge 0$, where $G_{[0,\infty)}$ is defined in \eqref{def:m-M}, then $\psi\ge 0$.
\end{thm}

\begin{proof}
Let $E_p = E_H((-p, p))$, where $E_H$ is the spectral measure of $H$ and $p\in\N$, and define
			\[
			H_p := H E_p, \qquad V_p := E_p V E_p.
			\]
			Then $H_p, V_p$ are self-adjoint operators acting on the Hilbert space $ E_p(\Hcal)$.
			Since $E_p \uparrow I$, it follows that $\|V_p - V\|_n \to 0$ and $H+tV_p \to H+tV$ resolvent strongly as $p \to \infty$ for each $t\geq 0$.

Consider the sequence of functions
 $$\psi_p(s)=\Tr\Big(\frac{d^n}{dt^n}\, g\big(E_p(H + tV)\big|_{E_p(\Hcal)}\big)\Big|_{t=s}\Big),\quad p\in\N.$$
By the same reasoning as in the proof of Theorem \ref{thm:Ext-otte}(i),
$$\psi_p(s)=\Tr\Big(\frac{d^n}{dt^n} g(H_p + tV_p)\Big|_{t=s}\Big),\quad s\in[0,\infty),\quad p\in\N.$$

By Theorem \ref{thm:differetiation} and \cite[Lemma 2.3]{LeSk20}, for every $s\geq0$ we have
\begin{align}
\label{seceq}
				\frac{d^n}{dt^n} g(H_p + tV_p)\Big|_{t=s}
				&= n! \, T_{g^{[n]}}^{(H_p + sV_p)^{n+1}}\!\big((V_p)^n\big)
				= n! \, T_{g^{[n]}}^{(H + sV_p)^{n+1}}\!\big((V_p)^n\big),
\end{align}
which, by Lemma~\ref{lem:moi-cont}, converges to
\[n! \,T_{g^{[n]}}^{(H + sV)^{n+1}}\!\big((V)^n\big)=\frac{d^n}{dt^n}g(H+tV)\Big|_{t=s}\]
in the $\|\cdot\|_1$-norm. Hence, $\psi_p(s)\rightarrow\psi(s)$ for all $s\geq0$.

Applying the argument of Theorem \ref{thm:Ext-otte} to the function $g|_{G_{[0,\infty)}}$ and bounded operators $E_pH|_{E_p(\Hcal)}$ and $E_p V|_{E_p(\Hcal)}$ gives $\psi_p\ge 0$ for every $p\in\N$. Hence, $\psi\ge 0$.
\end{proof}

	\section{Positivity preserving property of spectral shift functions}
	\label{sec_pos_ssf}
	
	In this section, we establish the sign-definiteness of higher order spectral shift functions $\eta_{n,H,V}$ satisfying Theorem~\ref{thm:existancessf}. Since the result is known in the cases $n=1,2$ (see Section \ref{sec:intro}), we prove it only in the case $n\ge 3$.

	\begin{thm}\label{thm:positivity-ssf}
		Let $n\in\N$ and let $H, V$ be self-adjoint operators in $\Hcal$ such that $V \in \Scal^n$. Let $\eta_{n, H, V}$ be given by Theorem \ref{thm:existancessf}. Then, the following assertions hold.
		\begin{enumerate}[(i)]
			\item\label{bdd_even_ssf} If $n$ is even, then $\eta_{n, H, V} \geq 0$.
			\item\label{bdd_odd_ssf} If $n$ is odd and $V \ge 0$ (respectively, $V\leq 0$), then $\eta_{n, H, V} \ge 0$ (respectively, $\eta_{n, H, V}\leq 0$).
		\end{enumerate}
	\end{thm}
	
	\begin{proof}
 Let $n\ge 3$.

 Step 1: $H$ is bounded.
		
	\eqref{bdd_even_ssf}	Let $n$ be even, and set $M = \|H\| + \|V\|$.
	Let $f \in C^{n+1}([-M,M])$.
	By~\cite[Theorem~5.4.3]{ST19}, we have
	\begin{align}\label{eq:trace_id_0}
		\Tr\!\big(\Rcal_n(f,H,V)\big)
		&= \frac{1}{(n-1)!}
		\int_0^1 (1-s)^{n-1}
		\Tr\!\left(
		\frac{d^n}{dt^n} f(H + tV)\Big|_{t=s}
		\right) ds.
	\end{align}
	Combining \eqref{eq:trace_id_0} with \eqref{traceformula_inv} implies
	\begin{align}\label{eq:trace_id-1}
		\int_{-M}^{M}
		f^{(n)}(\lambda)\, \eta_{n,H,V}(\lambda)\, d\lambda
		&= \frac{1}{(n-1)!}
		\int_0^1 (1-s)^{n-1}
		\Tr\!\left(
		\frac{d^n}{dt^n} f(H + tV)\Big|_{t=s}
		\right) ds.
	\end{align}
	By Theorem~\ref{thm:Ext-otte} and~\eqref{eq:trace_id-1}, it follows that
	\begin{align}
		\int_{-M}^{M}
		f^{(n)}(\lambda)\, \eta_{n,H,V}(\lambda)\, d\lambda
		\ge 0
	\end{align}
	for every $f \in C^{n+1}([-M,M])$ satisfying $f^{(n)} \ge 0$.
	Hence, $\eta_{n,H,V} \ge 0$.
	
	The proof of \eqref{bdd_odd_ssf} is completely analogous and, hence, omitted.
	
	\medskip
	
 Step 2: $H$ is unbounded.
	
		Let $E_k = E_H((-k,k))$, where $E_H$ is the spectral measure of $H$, and define
		\[H_k := H E_k, \quad V_k := E_k V E_k.\]
We have $E_k\uparrow I$ and $\|V_k-V\|_n\rightarrow 0$ as $k\rightarrow\infty$.
		
		Denote $\eta_n := \eta_{n, H, V}$ and $\eta_{n,k} := \eta_{n, H_k, V_k}$.
		Firstly we prove that
		\begin{align}
			\label{L1cv}
			\lim_{k \to \infty} \|\eta_n - \eta_{n,k}\|_{L^1(\R)} = 0.
		\end{align}
		Let $f \in C_c^{n+1}(\R)$.  Applying Lemma~\ref{lem:pert_form} and \eqref{a1} (analogously to the derivation of the second equality in \eqref{seceq}) yields
		\begin{align*}
			&\Rcal_n(f, H, V) - \Rcal_n(f, H_k, V_k)\\
			=& \Big(T_{f^{[n-1]}}^{H, H+V, (H)^{n-2}}((V)^{n-1}) - T_{f^{[n-1]}}^{(H)^n}((V)^{n-1})\Big) \\
			&\hspace*{1.5in}
			- \Big(T_{f^{[n-1]}}^{H_k, H_k+V_k, (H_k)^{n-2}}((V_k)^{n-1}) - T_{f^{[n-1]}}^{(H_k)^n}((V_k)^{n-1})\Big) \\
			=& \Big(T_{f^{[n-1]}}^{H, H+V, (H)^{n-2}}((V)^{n-1}) - T_{f^{[n-1]}}^{H, H+V_k, (H)^{n-2}}((V_k)^{n-1})\Big) \\
			&\hspace*{1.5in}- \Big(T_{f^{[n-1]}}^{(H)^n}((V)^{n-1}) - T_{f^{[n-1]}}^{(H)^n}((V_k)^{n-1})\Big).
		\end{align*}
		Applying Lemma \ref{lem:moi_pert_form} along with telescoping in the above equality and then applying Lemma~\ref{lem:moi_pert_form} one more time yields
	\begin{align}
		\label{eq:rem_diff}
		\nonumber
		&\Rcal_n(f, H, V) - \Rcal_n(f, H_k, V_k)\\
		\nonumber
		=& T_{f^{[n]}}^{H, H+V, H+V_k, (H)^{n-2}}(V, V-V_k, (V)^{n-2}) \\
		\nonumber
		&+\sum_{i=1}^{n-1} \Big(T_{f^{[n-1]}}^{H, H+V_k, (H)^{n-2}}((V_k)^{i-1}, V-V_k, (V)^{n-i-1})- T_{f^{[n-1]}}^{(H)^n}((V_k)^{i-1}, V-V_k, (V)^{n-i-1}) \Big) \\
		=& T_{f^{[n]}}^{H, H+V, H+V_k, (H)^{n-2}}(V, V-V_k, (V)^{n-2})\\ \nonumber
&+T_{f^{[n]}}^{H, H+V_k, (H)^{n-1}}(V-V_k,V_k,(V)^{n-2})+\sum_{i=2}^{n-1} T_{f^{[n]}}^{H, H+V_k, (H)^{n-1}}((V_k)^i, V-V_k, (V)^{n-i-1}).
	\end{align}
	Applying Corollary~\ref{thm:tr-est} in \eqref{eq:rem_diff} yields
	\begin{align*}
		\left|\Tr\big(\Rcal_n(f, H, V)\big) - \Tr\big(\Rcal_n(f, H_k, V_k)\big)\right|
		\leq n\, c_{n}\, \|f^{(n)}\|_\infty \, \|V\|_n^{\, n-1} \, \|V - V_k\|_n.
	\end{align*}
	The latter along with Theorem~\ref{thm:existancessf} implies
	\begin{align*}
		\sup_{\substack{f \in C_c^{n+1}(\R)\\ \|f^{(n)}\|_\infty \leq 1}}
		\left| \int_\R f^{(n)}(\lambda) \, (\eta_n - \eta_{n,k}) \, d\lambda \right|
		\leq n \, c_{n}\, \|V\|_n^{\, n-1} \, \|V - V_k\|_n \to 0
	\end{align*}
	as $k \to \infty$, confirming \eqref{L1cv}.
	
	Since $\|\eta_n - \eta_{n,k}\|_{L^1(\R)} \to 0$ as $k \to \infty$, there exists a subsequence $\{\eta_{n,k_l}\}_{l=1}^\infty$ converging to $\eta_n$ almost everywhere.
	If $n$ is even, by Step 1, $\eta_{n,k_l} \geq 0$ for all $k_l$, so we have $\eta_n \geq 0$. For odd $n$, the assumption $V \geq 0$ (respectively, $V\leq 0$) ensures $V_k \geq 0$ (respectively, $V_k\leq 0$), and a similar argument establishes the desired result.
	\end{proof}

\section{Infinite-dimensional BMV conjecture}
\label{sec:bmv}

In this section we obtain results on the BMV conjecture for self-adjoint operators with essential spectra.

We recall that a nonnegative function $f\in C^\infty([0,\infty))$ is said to be completely monotone if
\[	(-1)^n f^{(n)}(t)\ge 0 \text{ for all } n\in\N \text{ and } t>0,\]
and is said to be a Bernstein function if
\[(-1)^{\,n-1} f^{(n)}(t)\ge 0 \text{ for all } n\in\N \text{ and } t>0.\]

The existence of the representation \eqref{lir}, which proves the finite-dimensional BMV conjecture for traces, is equivalent to the complete monotonicity of the function $t\mapsto\Tr(e^{H-tV})$ (see, e.g., \cite[Theorem 1.4]{berstein_book}). The trace of $e^{H-tV}$ is generally undefined for infinite-dimensional $H$ and $V$, and subtracting $e^H$ as a natural correcting term inside the trace when $H$ is bounded from above and $0\leq V\in\Scal^1$ produces a function that is generally not completely monotone, but whose negative derivative is. The latter explains a modification of the integral representation \eqref{lir} obtained below in the case ${\rm dim}(\Hcal)=\infty$.

			\begin{thm}\label{thm:BMV-unbdd}
		Let $m\in\R$ and $H, V$ be self-adjoint operators in $\Hcal$ satisfying $H\geq  mI$ and $V\in\Scal^1$, $V\ge0$. Let $f\in C^\infty(G_{[0,\infty)})$, where $G_{[0,\infty)}$ is defined in \eqref{def:m-M}, satisfy $(-1)^{k-1}f^{(k)}\ge0$ for all $k\in\N$.  If $H$ is unbounded assume also that $f$ admits an extension to a function in $\cap_{j=0}^\infty \Wcal_j(\R)$. Then, there exist $b \geq 0$ and a positive measure $\mu$ on $(0,\infty)$ satisfying
			\[\int_{(0,\infty)} (1 \wedge s) \, d\mu(s) < \infty\]
			such that
			\begin{align}
\label{fBt}
				\Tr\big(f(H+tV) - f(H)\big) = b\,t + \int_{(0,\infty)} \big(1-e^{-ts}\big)\, d\mu(s)
			\end{align}
			for every $t \geq 0$. The pair $(b,\mu)$ is uniquely determined by the function on the left-hand side of \eqref{fBt}.
		\end{thm}

		\begin{proof}
Let $x\geq 0$ and let $g\in \cap_{j=0}^\infty \Wcal_j(\R)$ be such that $g|_{G_{[0, x+1)}}=f|_{G_{[0, x+1)}}$. If $H$ is unbounded, such $g$ exists by the assumption of the theorem; if $H$ is bounded, such $g$ exists by the observation made at the beginning of the proof of Theorem \ref{thm:Ext-otte}.
			
By \cite[Consequence of Theorem 4.5]{BiSo73} (alternatively, see \cite[Theorem 3.3.8]{ST19}), for each $t\in [0,x+1)$ we have
			\begin{align*}
				f(H+tV) - f(H) = g(H+tV) - g(H) = T_{g^{[1]}}^{H+tV, H}(tV) \in \Scal^1.
			\end{align*}
Hence, the function
\begin{align}
\label{phidef}
\phi(t) = \Tr\big(f(H + tV) - f(H)\big)
\end{align}
is defined for every $t \in [0,x+1)$ and
$$\phi(t) = \Tr\,\big(g(H + tV) - g(H)\big),\quad t\in [0,x+1).$$
By Theorem \ref{thm:differetiation}, for all $s\in[0,x+1)$ and $n \in \N$,
$$\phi^{(n)}(s)=\Tr\Big(\frac{d^n}{dt^n}g(H+tV)\Big|_{t=s}\Big).$$
Hence, by Theorem~\ref{thm:Ext-otte} when $H$ is bounded and Theorem~\ref{thm:Ext-otte-unbdd} otherwise, $(-1)^{n-1}\phi^{(n)}(s)\ge 0$ for all $s\in(0,x+1)$ and $n \in \N$.

Applying the above reasoning to every $x\ge 0$ implies that the definition \eqref{phidef} of $\phi$ extends to all $t\in [0,\infty)$, that $\phi\in C^\infty([0,\infty))$, and
\[(-1)^{n-1}\phi^{(n)}(s)\ge 0,\quad s\in(0,\infty),\; n \in \N.\] Since $\eta_{1,H, tV}\ge 0$ by Theorem \ref{thm:positivity-ssf}, it follows from $f'\ge 0$ and the representation \eqref{traceformula_inv} that $\phi\ge 0$. Thus, $\phi$ is a Bernstein function.
Applying \cite[Theorem 3.2]{berstein_book} to
$\phi$ completes the proof.
\end{proof}

\begin{cor}\label{cor:BMV-1}
Let $m,\lambda,r\in\R$,  $r\ge 1$, $\lambda<m$, and let $H,V$ be self-adjoint  operators in $\Hcal$ such that  $H\ge mI$ and $V\in\Scal^1$, $V\ge0$. Then there exist $b_1,b_2\le0$ and positive measures $\mu,\nu$ on $(0,\infty)$ satisfying
\[\int_{(0,\infty)} (1 \wedge s) \, d\mu(s)<\infty\;\, \text{ and }\; \int_{(0,\infty)} (1 \wedge s) \, d\nu(s) < \infty\]
such that
\begin{align}
\label{eq:bmv-1}&\Tr\big(e^{-H-tV} - e^{-H}\big) = b_1 t + \int_{(0,\infty)} \big(e^{-ts}-1\big)\, d\mu(s),\\
\label{eq:bmv-2}&\Tr\big((H+tV-\lambda I)^{-r}-(H-\lambda I)^{-r}\big) = b_2 t + \int_{(0,\infty)} \big(e^{-ts}-1\big)\, d\nu(s)
\end{align}
for every $t \geq 0$.  The pair $(b_1,\mu)$ is uniquely determined by the function on the left-hand side of \eqref{eq:bmv-1} and the pair $(b_2,\nu)$ is uniquely determined by the function on the left-hand side of \eqref{eq:bmv-2}.
\end{cor}

\begin{proof}
On $[m,\infty)$, consider the functions $f_1(x)=-e^{-x}$, and  $f_2(x)=-(x-\lambda)^{-r}$. Clearly, $f_i\in C^\infty([m,\infty))$ and $f_i^{(k)}\in L^2([m,\infty))$ for all $k\in\N_0$, $i=1,2$. Let $g_i\in C^{\infty}(\R)$, $i=1,2$, be such that
\begin{align}
g_i(x)=\begin{cases}
0 & \text{ if } x\leq \frac{ m+\lambda}{2}\\
f_i(x) & \text{ if } x\geq m
\end{cases}
\end{align}
for $i=1, 2$. Then, $g_i\in\cap_{j=0}^\infty \Wcal_j(\R)$, $i=1,2$.
Applying Theorem~\ref{thm:BMV-unbdd} to $f_1(x)$ gives \eqref{eq:bmv-1}, and to $f_2(x)$ gives \eqref{eq:bmv-2}.
\end{proof}

The next result extends \cite[Theorem 2]{LiSe12} to the infinite-dimensional setting and non-trace-class perturbations. It is obtained by applying Theorem~\ref{thm:Ext-otte} (if $H$ is bounded) or Theorem~\ref{thm:Ext-otte-unbdd} (if $H$ is unbounded) to $f(x) = (x-\lambda)^p$ on $[m, \infty)$. If $p\in\N$, then the assumption on $\lambda$ in Proposition \ref{cor:BMV-0} below can be relaxed from $\lambda<m$ to $\lambda\le m$.

\begin{ppsn}\label{cor:BMV-0}
	Let $p, \lambda, m\in  \R, \lambda<m$, and $k\in\N$.  Let $H, V $ be self-adjoint operators in $\Hcal$ satisfying $H\geq mI$, $V\ge 0$ and $V\in \Scal^k$. Then, the following assertions hold.
	\begin{enumerate}[(i)]
		\item If $H\in\bh$, $p>0$, and $1 \le k \le \lceil p \rceil$, then\; $\Tr\big(\frac{d^k}{dt^k}(H+tV-\lambda I)^p\big|_{t=s}\big)\geq 0$\; for all $s\geq 0$.\\
		
		\item If $H\in\bh$, $p > 0$, and $k \ge \lceil p \rceil$, then\; $(-1)^{k-\lceil p \rceil}\Tr\big(\frac{d^k}{dt^k}(H+tV-\lambda I)^p\big|_{t=s}\big)\geq 0$\; for all $s\geq 0$.\\
		
		\item  If $p<0$,
then\; $(-1)^k\,  \Tr\big(\frac{d^k}{dt^k}(H+tV-\lambda I)^p\big|_{t=s}\big)\geq 0$\; for all $s\geq 0$.
	\end{enumerate}
\end{ppsn}

A particular instance of Proposition \ref{cor:BMV-0}(i) ensures that if $\dim(\Hcal)<\infty$, $H\geq 0$, and $V\ge 0$, then for all $n,k\in\N$, $$\Tr\!\left(\frac{d^k}{dt^k}(H+tV)^n\Big|_{t=0}\right)
=\frac{d^k}{dt^k}\Tr\big((H+tV)^n\big)\Big|_{t=0}\ge0.$$
The latter nonnegativity of the derivatives is equivalent to the property that the polynomials $t\mapsto\Tr((H+tV)^n)$ have only nonnegative coefficients,  which, by \cite[Theorem 1]{LiSe04}, is equivalent to the existence of the representation \eqref{lir}.

 The following result generalizes the complete monotonicity of the trace of higher order operator derivatives obtained in parts (ii) and (iii) of Proposition \ref{cor:BMV-0}.

\begin{thm}
\label{bmvsn}
Let $m\in\R$, $n\in\N$, and $H, V$ be two self-adjoint operators in $\Hcal$ satisfying $H\geq mI$ and $V\in\Scal^n$, $V\ge0$. Let $f\in C^\infty(G_{[0,\infty)})$, where $G_{[0,\infty)}$ is defined in \eqref{def:m-M}, satisfy $(-1)^{k-1}f^{(k)}\ge0$ for all $k\in\N$ such that $k\ge n$. If $H$ is unbounded assume also that $f$ admits an extension to a function in $\cap_{j=1}^\infty\Wcal_j(\R)$. Then,
there exists a finite measure $\mu$ on $[0,\infty)$ satisfying $(-1)^{n-1}\mu\ge 0$ such that
\begin{align}\label{fBtn}
\Tr\!\left(	\frac{d^n}{d\tau^n} f(H + \tau V)\Big|_{\tau=t}\right) = \int_{[0,\infty)}\, e^{-ts}\,d\mu(s)
\end{align}
for every $t\ge0$. The measure $\mu$ is uniquely determined by the function on the left-hand side of \eqref{fBtn}.
\end{thm}

	\begin{proof}
 Let $x\ge 0$ and let $g\in \cap_{j=1}^\infty \Wcal_j(\R)$ be such that  $g|_{G_{[0,x+1)}}=f|_{G_{[0,x+1)}}$. By Theorem \ref{thm:differetiation} applied to $g$ and by the equality $f(H+tV)=g(H+tV)$ for all $t\in [0,x+1)$, the function	
\begin{align*}
\phi(t)=\Tr\!\left(\frac{d^n}{d\tau^n} f(H + \tau V)\Big|_{\tau=t}\right)
\end{align*}
is well defined for all $t\in [0,x+1)$.
By Theorem \ref{thm:differetiation} , we have
\begin{align*}
\phi^{(k)}(t)=\Tr\!\left(
\frac{d^{n+k}}{d\tau^{n+k}} f(H + \tau V)\Big|_{\tau=t}\right), \quad t\in [0,x+1).
\end{align*}
Combining the latter with Theorem~\ref{thm:Ext-otte} when $H$ is bounded and
Theorem~\ref{thm:Ext-otte-unbdd} otherwise implies that $\phi^{(k)}$ has the same sign as $f^{(n+k)}$ for $k\in\N_0$. Therefore,
\begin{align}
\label{cdercm}
(-1)^k \big((-1)^{n-1}\phi\big)^{(k)}(t)\geq 0
\end{align}
for all $t\in(0,x+1)$ and all $k\in\N_0$.

Applying the above reasoning to every $x\ge 0$ implies that the definition of $\phi$ extends to $[0,\infty)$ and that \eqref{cdercm} holds for all $t>0$. Thus, $(-1)^{n-1}\phi$ is a completely monotone function. By Bernstein's theorem \cite[Theorem 1.4]{berstein_book}, the latter property is equivalent to the existence of the unique measure  $\mu$ satisfying \eqref{fBtn}. The finiteness of $\mu$ follows from \cite[Proposition 1.2]{berstein_book}.
	\end{proof}

 By the method developed in this section, the time-dependent higher order Taylor remainder $t\mapsto\Tr\big(\Rcal_n(f,H+tV,V)\big)$ can be expressed as the Laplace transform of a time-independent finite measure, whereas the alternative representation provided in Theorem~\ref{thm:existancessf} yields a measure that depends on $t$. Under the assumptions of Theorem~\ref{bmvsn}, we obtain
\begin{align}
\label{fRtn}
\Tr\big(\Rcal_n(f,H+tV,V)\big)=\int_{[0,\infty)}\, e^{-ts}\,d\nu(s),\quad t\geq 0,
\end{align}
where the measure $\nu$ satisfies $(-1)^{n-1}\nu\ge 0$ and is uniquely determined by the function on the left-hand side of \eqref{fRtn}. 
\smallskip

\smallskip
	
	\noindent\textit{Acknowledgment}: C. Pradhan gratefully acknowledges support from the Fulbright-Nehru Postdoctoral Fellowship, Award No. 3130/FNPDR/2025. A. Skripka is supported in part by Simons Foundation Grant MP-TSM-00002648.

\end{document}